\documentclass[a4paper,11pt,reqno]{amsart}
\usepackage{amsmath,amsfonts,amssymb}
\usepackage[style=numeric]{biblatex}
\usepackage{verbatim}
\usepackage{enumerate}
\usepackage{url}
\usepackage[colorlinks]{hyperref}
\usepackage[latin1]{inputenc}
\usepackage{enumitem}
\usepackage[english]{babel}
\usepackage{tikz}
\usepackage[margin=3.5cm]{geometry}
\usepackage{bbm}
\usepackage{multirow}
\usepackage{mathrsfs}
\usepackage{caption}
\usepackage{float}
\restylefloat{table}
\usepackage{graphicx}

\def\Z{\mathbb Z}

\def\Ca{\mathcal C}

\def \F{\mathbb F}

\def\Fq{\mathbb{F}_q}
\def\Fp{\mathbb{F}_p}
\def\Fqn{\mathbb{F}_{q^n}}

\def\Fqdo{\mathbb{F}_{q^2}}
\def\Fpti{\mathbb{F}_{p^{t_i}}}

\DeclareMathOperator{\tr}{Tr}

\theoremstyle{plain}
\newtheorem{theorem}{Theorem}[section]
\newtheorem{lemma}[theorem]{Lemma}
\newtheorem{definition}[theorem]{Definition}
\newtheorem{corollary}[theorem]{Corollary}
\newtheorem{proposition}[theorem]{Proposition}
\newtheorem{conjecture}[theorem]{Conjecture}
\newtheorem{remark}[theorem]{Remark}

\newtheorem{problem}[theorem]{Problem}

\theoremstyle{definition}

\author[J. A. Oliveira]{Jos\'e Alves Oliveira}
\address{Departamento de Matem\'{a}tica,
Universidade Federal de Minas Gerais,
UFMG,
Belo Horizonte MG (Brazil),
 30123-970}

\email{joseufmg@gmail.com}

\date{\today
}

\title{Diagonal equations with restricted solution sets}

\begin{document}

\begin{abstract}
Let $\mathbb{F}_q$ be a finite field with $q=p^n$ elements. In this paper, we study the number of solutions of equations of the form $a_1 x_1^{d_1}+\dots+a_s x_s^{d_s}=b$ with $x_i\in\mathbb{F}_{p^{t_i}}$, where $b\in\mathbb{F}_q$ and $t_i|n$ for all $i=1,\dots,s$. In our main results, we employ results on quadratic forms to give an explicit formula for the number of solutions of diagonal equations with restricted solution sets satisfying certain natural restrictions on the exponents. As a consequence, we present conditions for the existence of solutions. We also discuss further questions concerning equations with restricted solution sets and present some open problems.
\end{abstract}
\maketitle

\section{Introduction}

Let $p$ be a prime number and $t$ be a positive integer. Let $\Fq$ be a finite field with $q$ elements, where $q=p^n$.  For $\vec{a}=(a_1,\dots,a_s)\in\Fq^s$, $\vec{d}=(d_1,\dots,d_s)\in\Z_+^s$, $\vec{t}=(t_1,\dots,t_s)\in\Z_+^s$ and $b\in\Fq$, let $N_s(\vec{a},\vec{d},\vec{t},q,b)$ be the number of solutions of the diagonal equation
\begin{equation}\label{item300}
a_1 x_1^{d_1}+\cdots+a_s x_s^{d_s}=b
\end{equation}
where $x_i\in\Fpti$. In the case $t_1=\dots=t_s=n$,  Weil~\cite{weil1949numbers} and Hua and Vandiver~\cite{hua1949characters} independently showed that $N_s(\vec{a},\vec{d},q,b)$ can be expressed in terms of character sums. In particular, Weil's result implies that
\begin{equation}\label{item301}
|N_s(\vec{a},\vec{d},\vec{t},q,0)-q^{s-1}|\le I(d_1,\dots,d_s)(q-1)q^{(s-2)/2}
\end{equation}
where $I(d_1,\dots,d_s)$ is the number of $s$-tuples $(y_1,\dots,y_s)\in\Z^n$, with $1\leq y_i\leq d_i-1$ for all $i=1,\dots,s$, such that
\begin{equation}\label{item302}
\frac{y_1}{d_1}+\dots+\frac{y_s}{d_s}\equiv 0\pmod{1}.
\end{equation}
 The study of solutions of diagonal equations played an essential role in the statement of Weil's conjectures for algebraic varieties and in the development of the algebraic geometry. The number of solutions of diagonal equations have been extensively studied in the last few decades \cite{cao2007factorization,baoulina2010number,hou2009certain,sun1996number,qi1997diagonal,oliveira2019rational,baoulina2016class,zhou2019counting}. In many cases, the authors present a formula for the number of solutions of equations whose exponents satisfy certain natural restrictions. The case where $q$ is a square provides families of diagonal equations whose number of solutions can be obtained by means of simple parameters ~\cite{wolfmann1992number,cao2016number}. Recently, we presented \cite{oliveira2020diagonal} a formula for the number of solutions of Eq.~\eqref{item300} in the case where $t_1=\dots=t_s=n$, $q=p^{2n}$ and  there exists a divisor $r$ of $n$ such that $d_i|(p^r+1)$ for all $i=1,\dots,s$. Most notably, in \cite{oliveira2020diagonal} we studied the number of solutions of \eqref{item300} in order to find those diagonal equations whose number of solutions attains the bound \eqref{item301}.  For more results concerning diagonal equations over finite fields, see Section 7.3 in \cite{Panario} and the references therein.
 
 In \cite{reis2020counting}, the author studies a class of suitable generalized diagonal equations with restricted solutions sets. Inspired by the results of \cite{reis2020counting} and by the results obtained recently in \cite{oliveira2020diagonal}, this paper's aim is to generalize the results of~\cite{oliveira2020diagonal} to diagonal equations with restrict solution sets.
 In order to do that, we employ some well-known results on quadratic forms over finite dimensional vector spaces over a finite field. Also, we propose some open problems and present conjectures on the number of solutions of affine varieties.  
 
 As follows, we state our main results and some important remarks. Throughout the paper, unless otherwise stated,  $q=p^{2n}$ for some positive integer $n$ and an odd prime $p$. Let $n$ be a positive integer. For $t$ a divisor of $n$, let $\tr_{p^n,p^{t}}$ denote the trace function from $\F_{p^n}$ to $\F_{p^t}$. Throughout the paper, unless otherwise stated, we let $\vec{a}\in\Fq^s$, $t=(2t_1,\dots,2t_s)$ with $t_i|n$ for all $i=1,\dots,s$. Since we are interested in the number of solutions of Eq.~\eqref{item300}, by a simple change of variables, we may assume without loss of generality that $d_i$ is a divisor of $p^{2t_i}-1$ for all $i=1,\dots,s$.
 \begin{definition}\label{item312}
 	Let $a\in\Fq$ and let $d,t,r$ be positive integers such that  and $r|t$, $t|n$ and $d|(p^{2t}-1)$. We set $m=\tfrac{p^{2t}-1}{d}$, $u=\tfrac{p^r+1}{d}$ and $\varepsilon=(-1)^{t/{r}}$. Let $T(x)=\tr_{q,p^{2t}}(x)$.
 	\begin{enumerate}[label=(\alph*)]
 		\item For $c\in\Fq$, let
 		$${\small \Delta(a,d,t,r,c)=\begin{cases}
 			p^{t}, &\text{ if }T(ca)=0;\\
 			\varepsilon(1-d), &\text{ if }T(ca)^{m}=\varepsilon^{u};\\
 			\varepsilon, &\text{ if }T(ca)^{m}\not\in\{0, \varepsilon^{u}\};\\
 			\end{cases}}$$
 		\item Let $\Lambda(a,t)=\{c\in\Fq:T(ca)=0\}$.
 	\end{enumerate}
 \end{definition}
 
 The main results of the paper are the following.

 \begin{theorem}\label{item303}
 	Assume that for each $i$, with $i=1,\dots,s$, there exists a divisor $r_i$ of $t_i$ such that $d_i|(p^{r_i}+1)$. Then 
 	$$N_s(\vec{a},\vec{d},\vec{t},q,0)=p^{t_1+\cdots+t_s-2n}\sum_{c\in\Fq^*}\prod_{i=1}^{s} \Delta(a_i,d_i,t_i,r_i,c),$$
 	where $\Delta(a_i,d_i,t_i,r_i,c)$ is defined as in Definition~\ref{item312}.
 
 \end{theorem}

\begin{remark}
	As a direct consequence of Theorem~\ref{item303}, we have that $p^{t_1+\cdots+t_s-2n}$ divides $N_s(\vec{a},\vec{d},\vec{t},q,0)$. Therefore, if $N_s(\vec{a},\vec{d},\vec{t},q,0)>0$, then $N_s(\vec{a},\vec{d},\vec{t},q,0)\geq p^{t_1+\cdots+t_s-2n}$.
\end{remark}

 \begin{corollary}\label{item309}
	Let $b\in\Fq^*$ and suppose $d_1=\dots=d_s=d$. Assume that for each $i$, with $i=1,\dots,s$, there exists a divisor $r_i$ of $t_i$ such that $d|(p^{r_i}+1)$. Let $t_l$ be the least integer such that $p^{t_l}\equiv -1\pmod{d}$. Then 
	$$N_s(\vec{a},\vec{d},\vec{t},q,b)=\frac{p^{t_1+\cdots+t_s-2n}}{p^{2t_l}-1}\sum_{c\in\Fq^*}\left(\Delta(-b,d,t_l,t_l,c)p^{t_l} -1\right)\prod_{i=1}^{s}\Delta(a_i,d,t_i,r_i,c),$$
	where $\Delta(-b,d,t_l,t_l,c)$ and $\Delta(a_i,d,t_i,r_i,c)$ are defined as in Definition~\ref{item312}.
\end{corollary}

Many results in the literature (for example Theorem 1 of \cite{wolfmann1992number}, Theorem 2.9 of \cite{cao2016number} and Theorem 2.3 of \cite{oliveira2020diagonal}) can be obtained by a straight employment of Theorem~\ref{item303} and Corollary~\ref{item309}, which means that the results of this paper are far more general than the known results. The case where $d_1=\dots=d_s=2$ and $n$ is even is covered by Corollary~\ref{item309}. More about the case where $d_1=\dots=d_s=2$ is discussed in Section~\ref{item334}.

\begin{corollary}\label{item313}  Assume that for each $i$, with $i=1,\dots,s$, there exists a divisor $r_i$ of $t_i$ such that $d|(p^{r_i}+1)$. Then 
	$$\left|\frac{N_s(\vec{a},\vec{d},\vec{t},q,0)}{ p^{t_1+\cdots+t_s-2n}}-|\Lambda|p^{t_1+\cdots+t_s}\right|\leq \sum_{c\in\Fq\backslash\Lambda}\ \prod_{i\in \nu(c)} p^{t_i}\prod_{i\not\in \nu(c)}(d_i-1),$$
	where $\Lambda=\cap_{i=1}^{s}\Lambda(a_i,t_i)$ and $\nu(c)=\{i\in\{1,\dots,s\}:c\in\Lambda(a_i,t_i)\}$. In particular, if $\Lambda(a_i,t_i)\cap \Lambda(a_j,t_j)=\{0\}$ for all $1\leq i< j \leq s$ and
	$$\prod_{i=1}^s(d_i-1)\Bigg(\sum_{i=1}^s\bigg(\frac{1}{p^{t_i}(d_i-1)}-\frac{1}{p^{2t_i}}\bigg)+1+\frac{s-1}{q}\Bigg)<p^{t_1+\cdots+t_s-2n},$$
	 then there exists one solution for Equation~\ref{item300}. In this case, there exist at least $p^{t_1+\cdots+t_s-2n}$ solutions.
\end{corollary}

One can obtain bounds for the case $b\neq 0$ and $d_1=\dots=d_s$ from Corollary~\ref{item309}. In particular, tight bounds can be obtained in case where $s=2$, as we will see in Section~\ref{item325}.

\begin{corollary}\label{item311}
	 Let $t_l$ be the least integer such that $p^{t_l}\equiv -1\pmod{d}$. The number of solutions of the equation $a_1 x_1+\dots+a_s x_s=b$ with $x_i\in\F_{p^{2t_i}}$ is given by
	$$\frac{\lambda_b p^{2t_1+\cdots+2t_s+2t_l-2n}-\lambda p^{2t_1+\cdots+2t_s-2n}}{p^{2t_l}-1},$$
	where $\lambda:=|\cap_{i=1}^{s}\Lambda(a_i,t_i)|$ and $\lambda_b:=|\cap_{i=1}^{s}\Lambda(a_i,t_i)\cap\Lambda(-b,t_l)|$.
\end{corollary}

\begin{remark}
	In the particular case $b=0$, it follows that $\lambda_0=\lambda$ in Corollary~\ref{item311} and then the number of solutions of the equation $a_i x_1+\dots+a_s x_s=0$ with $x_i\in\F_{p^{2t_i}}$ equals $\lambda p^{2t_1+\cdots+2t_s-2n}$.
\end{remark}

The paper is organized as follows. Section~\ref{item332} provides preliminary results which will be used to prove the main results. In Section~\ref{item333}, we prove the main results of the paper. In Section~\ref{item325}, we study diagonal equations with two variable and give a bound for its number os solutions, generalizing Hasse-Weil bound for this class of equations. Section~\ref{item334} discuss the case $d_1=\dots=d_s=2$. Finally, in Section~\ref{item335} we provide some final considerations and open problems. 

\section{On the number of solutions of $\tr_{p^{2n},p}(ax^d)=\lambda$}\label{item332}

In this section we provide the tools from quadratic forms over finite fields which will be used throughout the paper. We use the ideas from Wolfman~\cite{wolfmann1989number} in order to generalize his results. The following result is well known (see \cite{Lidl}).

\begin{lemma}\label{item304}
	Let $\Phi$ be a quadratic form on a $k$ dimensional vector space $E$ over $\Fq$, $\Psi$ the associated symmetric bilinear form and $v$ the dimension of the kernel of $\Psi$. For $\lambda\in\F_q$, let $S_\lambda$ the number of solutions in $E$ of $\Phi(x)=\lambda$. If $k=2t$ and $v=2s$, then there exists $D\in\{-1,0,1\}$ such that
	$$S_\lambda=\begin{cases}
	q^{2t-1}+Dq^{t+s-1}(q-1),&\text{ if }\lambda=0;\\
	q^{2t-1}-Dq^{t+s-1},&\text{ if }\lambda\neq0.\\
	\end{cases}$$
\end{lemma}

\begin{proposition}\label{item308}
	Let $q=p^{2n}$ and let $t$ and $r$ be integers such that $t|n$ and $r|t$. Let $a\in\Fq$. Then the map $\Phi$ from $\F_{p^{2t}}$ into $\Fp$ defined by $\Phi_a(x)=\tr_{q,p}(ax^{p^r+1})$ is a quadratic form. Let $\hat{a}=\tr_{q,p^{2t}}(a)$ and $m=\tfrac{p^{2t}-1}{p^r+1}$. If $s$ and $D$ are integers as defined in Lemma~\ref{item304}, then
	\begin{enumerate}
		\item If $\hat{a}=0$, then $s=t$ and $D=1$;
		\item If $\hat{a}^m=(-1)^{t/r}$, then $s=r$ and $D=(-1)^{t/r-1}$;
		\item If $\hat{a}^m\not\in\{0,(-1)^{t/r}\}$, then $s=0$ and $D=(-1)^{t/r}$;
	\end{enumerate}
\end{proposition}

\begin{proof}
	We observe that $\Phi_a(x+y)=\Phi_a(x)+\Phi_a(y)+\tr_{q,p}(ax^{p^r}y+ay^{p^r}x)$ and $\tr_{q,p}(ax^{p^r}y+ay^{p^r}x)$ is a symmetric bilinear form.  Furthermore, if $\lambda\in\Fp$, then $\Phi_a(x)=\tr_{q,p}(ax^{p^r+1})$ and therefore $\Phi_a$ is a quadratic form. Now we compute the dimension of the kernel of $\tr_{q,p}(ax^{p^r}y+ay^{p^r}x)$. Let $\hat{a}=\tr_{q,p^{2t}}(a)$ and observe that
	$$\begin{aligned}
	\tr_{q,p}(ax^{p^r}y+ay^{p^r}x)&=\tr_{p^{2t},p}\left(\tr_{q,p^{2t}}(ax^{p^r}y+ay^{p^r}x)\right)\\
	&=\tr_{p^{2t},p}\left((x^{p^r}y+y^{p^r}x)\tr_{q,p^{2t}}(a)\right)\\
	&=\tr_{p^{2t},p}\left(x^{p^r}y\hat{a}\right)+\tr_{p^{2t},p}\left(y^{p^r}x \hat{a}\right)\\
	&=\tr_{p^{2t},p}\left(x^{p^r}y\hat{a}\right)+\tr_{p^{2t},p}\left(yx^{p^{2t-r}} \hat{a}^{p^{2t-r}}\right)\\
	&=\tr_{p^{2t},p}\left(y\left( x^{p^r}\hat{a}+x^{p^{2t-r}} \hat{a}^{p^{2t-r}}\right) \right),\\
	\end{aligned}$$
	which vanish for all $y$ if and only if $x^{p^r}\hat{a}+x^{p^{2t-r}} \hat{a}^{p^{2t-r}}=0$, which is equivalent to 
	\begin{equation}\label{item305}
	x^{p^{2r}}\hat{a}^{p^r}+x \hat{a}=0. 
	\end{equation}
	If $\hat{a}=0$, then $x^{p^r}\hat{a}+x^{p^{2t-r}} \hat{a}^{p^{2t-r}}=0$ for all $x\in\F_{p^{2t}}$ and so $s=t$. Otherwise, the nonzero solutions of Eq.~\ref{item305} are the elements $x\in\F_{q^t}$ such that $x^{p^{2r}-1}=-1\hat{a}^{1-p^r}$. The equation $x^{p^{2r}-1}=-1\hat{a}^{1-p^r}$ has solutions if and only if 
	\begin{equation}\label{item306}
	\left(-1\hat{a}^{1-p^r}\right)^{(p^{2t}-1)/(p^{2r}-1)}=1
	\end{equation}
	 and, in the affirmative case, there exist $p^{2r}-1$ distinct solutions. We observe that Eq.~\ref{item306} is equivalent to $\hat{a}^{\frac{p^{2t}-1}{p^r+1}}=(-1)^{t/r}$ and this proves our statements for kernel's dimension. Now, we compute the value of $d$. Let $L$ be the multiplicative subgroup of order $p^{r}+1$ of $\Fq^*$. If $c\in\F_{p^{2t}}$ is an element such that $\Phi_a(c)=\lambda$ for some $\lambda\in \F_p^*$, then $\Phi_a(cl)=\lambda$ for all $l\in L$. Therefore, Lemma~\ref{item304} entails that 
	 $$p^{2t-1}\equiv Dp^{t+s-1}\pmod{p^r+1},$$
	 which implies that $D=(-1)^{(t-s)/r}$, since we have already proved that $r|(t-s)$. This completes the proof of our proposition.
\end{proof}

\begin{remark}
	We can change $p$ by a power of prime without any further arguments in the proof. We use $p$ in our statements to make the proofs of our main results simpler. 
\end{remark}

\begin{theorem}\label{item307}
	Let $q=p^{2n}$ and let $t, r$ and $d$ be integers such that $t|n$,  $r|t$ and $d|(p^r+1)$. Let $a\in\F_{q}^*$, $\lambda\in\Fp$ and let $N_\lambda$ denote the number of solutions of $\tr_{q,p}\big(ax^{d}\big)=\lambda$ over $\F_{p^{2t}}$. Set $\hat{a}=\tr_{q,p^{2t}}(a)$, $k=\tfrac{p^{2t}-1}{d}$, $\varepsilon=(-1)^{t/r}$ and $u=\tfrac{p^r+1}{d}$. Then
	\begin{enumerate}[label=(\alph*)]
		\item $\hat{a}=0$ and $\lambda= 0$, $N_\lambda=p^{2t}$;
		\item $\hat{a}=0$ and $\lambda\neq 0$, $N_\lambda=0$;
		\item $\hat{a}^k\not\in \{0,\varepsilon^u\}$ and $\lambda= 0$, $N_\lambda=p^{2t-1}+\varepsilon p^{t-1}(p-1)$;
		\item $\hat{a}^k\not\in \{0,\varepsilon^u\}$ and $\lambda\neq 0$, $N_\lambda=p^{2t-1}+\varepsilon p^{t-1}$;
		\item $\hat{a}^k= \varepsilon^u$ and $\lambda\neq 0$, $N_\lambda=p^{2t-1}-\varepsilon p^{t-1}(d-1)(p-1)$;
		\item $\hat{a}^k= \varepsilon^u$ and $\lambda= 0$, $N_\lambda=p^{2t-1}-\varepsilon p^{t-1}(d-1)$.
	\end{enumerate}
\end{theorem}

\begin{proof} Let $\alpha$ be a primitive element of $\F_{p^{2t}}^*$. We observe that
	$$\big\{x\in\F_{p^{2t}}:\tr_{q,p}\big(ax^d\big)=\lambda\big\}=\bigcup_{\ell=0}^{\frac{p^r+1}{d}-1}\big\{x\in\F_{p^{2t}}:\tr_{q,p}\big(a\alpha^{\ell d} x^{p^r+1}\big)=\lambda\big\},$$
	where which element of the left-hand set has $\tfrac{p^r+1}{d}$ representatives in the right-hand side. In particular,
	\begin{equation}\label{item209}
	N_\lambda=|\{x\in\F_{p^{2t}}:\tr_{q,p}\big(ax^d\big)=\lambda\}|=\frac{d}{p^r+1}\left|\bigcup_{\ell=0}^{\frac{p^r+1}{d}-1}\{x\in\F_{p^{2t}}:\tr_{q,p}\big(a\alpha^{\ell d} x^{p^r+1}\big)=\lambda\}\right|.
	\end{equation}
	For $\ell=0,\dots,\tfrac{p^r+1}{d}-1$, we set $\hat{a}_\ell=\tr_{q,p^{2t}}(a\alpha^{\ell d})$ and observe that $\tr_{q,p^{2t}}(a\alpha^{\ell d})=\alpha^{\ell d}\tr_{q,p^{2t}}(a)=\alpha^{\ell d} \hat{a}$, so that $\hat{a}_\ell=0$ if and only if $\tr_{q,p^{2t}}(a)=0$. In this case, Proposition~\ref{item308} and Eq.~\eqref{item209} implies that $N_0=q^{2t}$ and $N_\lambda=0$ if $\lambda\neq 0$. Let $m=\tfrac{p^{2t}-1}{p^r+1}$ and $\varepsilon=(-1)^{t/r}$. We split the remaining part of the proof into two cases:
	\begin{itemize}
		\item If there exists no $j\in\{0,\dots,\tfrac{p^r+1}{d}-1\}$ such that $\hat{a}_j^m=(-1)^{t/r}$, then Proposition~\ref{item308} and Eq.~\eqref{item209} entail that 
		$$N_\lambda=\begin{cases}
		p^{2t-1}+\varepsilon p^{t-1}(p-1),&\text{ if }\lambda=0;\\
		p^{2t-1}-\varepsilon p^{t-1},&\text{ if }\lambda\neq 0.\\
		\end{cases}
		$$
		\item If there exists some $j\in\{0,\dots,\tfrac{p^r+1}{d}-1\}$ such that $\hat{a}_j^m=(-1)^{t/r}$, then it is easy to verify that $\hat{a}_i^m\neq (-1)^{t/r}$ for all $i\in\{0,\dots,\tfrac{p^r+1}{d}-1\}\backslash\{j\}$, so that a straightforward computation using Proposition~\ref{item308} and Eq.~\eqref{item209} shows that
		$$N_\lambda=\begin{cases}
		p^{2t-1}-\varepsilon p^{t-1}(d-1)(p-1),&\text{ if }\lambda=0;\\
		p^{2t-1}-\varepsilon p^{t-1}(d-1),&\text{ if }\lambda\neq 0.\\
		\end{cases}
		$$
	\end{itemize}
	 Our assertion follows by observing that there exists $j\in\{0,\dots,\tfrac{p^r+1}{d}-1\}$ such that $\hat{a}_j^m=(-1)^{t/r}$ if and only if $\hat{a}^{(p^{2t}-1)/d}=\varepsilon^{(p^r+1)/d}$.
\end{proof}

\section{Proof of the main results}\label{item333}
In this section, we prove our counting results. To prove Theorem \ref{item303}, we will follow the main ideas of Wolfmann~\cite{wolfmann1992number}. To this end, let $a\in\Fq$ and let $\psi_a(x)=\exp\left((2\pi i)\tr_{q,p}(ax)/p\right)$ be an additive character. We have the following known results.

\begin{lemma}\label{item324}
	Let $\vec{a}=(a_1,\dots,a_s)\in\Fq^s$ and $\vec{d}=(d_1,\dots,d_s)\in\Z_+^s$ such that $d_i|(q-1)$ for all $i=1,\dots,s$. Then 
	$$N_s(\vec{a},\vec{d},\vec{t},q,b)=q^{-1}\sum_{c\in\scalebox{0.7}{$\displaystyle \Fq$}} \psi_c(-b)\prod_{i=1}^{s}S_{i}(c),$$
	where $S_{i}(c):=\sum_{x\in\scalebox{0.7}{$\displaystyle \F_{p^{2t_i}}$}}\psi_{ca_i}\left(x^{d_i}\right)$.
\end{lemma}
\begin{proof}
	It follows by a argument similar to the proof of Proposition 1 of \cite{wolfmann1992number}.
\end{proof}

\begin{lemma}\label{item323}
	Let $q=p^{2n}$, $c\in\Fq$, and $\vec{a},\vec{d}$ and $S_{i}(c)$ as defined in Lemma~\ref{item324}. Suppose that there exists a divisor $r_i$ of $t_i$ such that $d_i|(p^{r_i}+1)$ for each $i=1,\dots,s$. Let $\varepsilon_i=(-1)^{t_i/r_i}$, $\epsilon_i=\varepsilon_i^{(p^{r_i}+1)/d_i}$ and $\hat{a}_i=\tr_{q,p^{2t_i}}(ca_i)$. Then, for each $i=1,\dots,s$, we have that
	\begin{enumerate}[label=(\alph*)]
		\item If $\hat{a}_i=0$, then $S_i(c)=p^{2t_i}$;\\
		\item If $\hat{a}_i^{(p^{2t_i}-1)/d_i}=\epsilon_i$, then $S_i(c)=-\varepsilon_i(d_i-1)p^{t_i}$;\\
		\item If $\hat{a}_i^{(p^{2t_i}-1)/d_i}\not\in\{0, \epsilon_i\}$, then $S_i(c)= \varepsilon_i\, p^{t_i} $.\\
	\end{enumerate}
\end{lemma}

\begin{proof}
	We observe that
	\begin{equation}\label{item328}
	S_{i}(c)=\sum_{\lambda\in\Fp} \exp\left(\frac{2\pi i\lambda}{p}\right)N_\lambda,
	\end{equation}
	where $N_\lambda$ is defined as in Lemma~\ref{item307}. Our result follows directly from Eq.~\eqref{item328} and Theorem~\ref{item307}.
\end{proof}

From Lemmas \ref{item324} and \ref{item323}, we are able to prove Theorem \ref{item303}.

\subsection{Proof of Theorem \ref{item303}} Since $b=0$, it follows from Lemma~\ref{item324} that 
	$$N_s(\vec{a},\vec{d},\vec{t},q,0)=q^{-1}\sum_{c\in\scalebox{0.7}{$\displaystyle \Fqdo$}} \prod_{i=1}^{s}S_{i}(c).$$
	Let $\varepsilon_i=(-1)^{t_i/r_i}$, $m_i=\tfrac{p^{2t_i}-1}{d_i}$ and $u_i=\tfrac{p^{r_i}+1}{d_i}$. Let $T_i(x)=\tr_{q,p^{2t_i}}(x)$. By Lemma~\ref{item323}, it follows that $S_{i}(c)=p^{t_i}\Delta(a_i,d_i,t_i,r_i,c)$, where
	$${\small \Delta(a_i,d_i,t_i,r_i,c)=\begin{cases}
		p^{t_i}, &\text{ if }T_i(ca_i)=0;\\
		\varepsilon_i(1-d_i), &\text{ if }T_i(ca_i)^{m_i}=\varepsilon_i^{u_i};\\
		\varepsilon_i, &\text{ if }T_i(ca_i)^{m_i}\not\in\{0, \varepsilon_i^{u_i}\}.\\
		\end{cases}}$$
	This completes the proof of our assertion.	$\hfill\qed$
	
\begin{definition}
	For $f\in\Fq[x_1,\dots,x_s]$ and $\vec{t}=(t_1,\dots,t_s)$ with $t_i|n$ for all $i=1,\dots,s$, let 
	$$V_{\vec{t}\,}(f(x_1,\dots,x_s)=0)=\{(x_1,\dots,x_s)\in\F_{t_1}\times\dots\times\F_{t_s}:f(x_1,\dots,x_s)=0\},$$
	the set of zeros of $f$. For $i\in\{1,\dots,s\}$, let
	$$V_{\vec{t},\, x_i}(f(x_1,\dots,x_s)=0)=\{(x_1,\dots,x_s)\in\F_{t_1}\times\dots\times\F_{t_s}:f(x_1,\dots,x_s)=0\text{ and }x_i\neq 0\}.$$
\end{definition}


\subsection{Proof of Corollary~\ref{item309}}
Let $b\in\Fq^*$ and let $x_1,\dots,x_s,y$ denote indeterminates. Let $\vec{t_0}=(2t_1,\dots,2t_s,2t_l)$ and $\vec{t}=(2t_1,\dots,2t_s)$. Let
 $$\psi: V_{\vec{t_0},\, y}(a_1 x_1^{d}+\cdots+a_s x_s^{d}=by^{d})\rightarrow V_{\vec{t}\, }(a_1 x_1^{d}+\cdots+a_s x_s^{d}=b)$$
 be the function defined by
$$(x_1,\dots,x_s,y)\mapsto(x_1/y,\dots,x_s/y).$$
Since $t_l|t_i$ for all $i=1,\dots,s$, it follows that $\psi$ is well-defined. It is direct to verify that $\psi$ is a $(p^{2t_l}-1)$-to-one function and then
\begin{equation}\label{item314}
|V_{\vec{t_0},\, y}(a_1 x_1^{d}+\cdots+a_s x_s^{d}=by^{d})|=(p^{2t_l}-1)|V_{\vec{t}\, }(a_1 x_1^{d}+\cdots+a_s x_s^{d}=b)|.
\end{equation}
We observe that $V_{\vec{t_0},\, y}(a_1 x_1^{d}+\cdots+a_s x_s^{d}=by^{d})$ is iqual to
$$ V_{\vec{t_0}\,}(a_1 x_1^{d}+\cdots+a_s x_s^{d}=by^{d})\backslash \left(V_{\vec{t_0}\,}(a_1 x_1^{d}+\cdots+a_s x_s^{d}=by^{d})\cap\{(x_1,\dots,x_s,y):y=0\}\right)$$
and 
$$|V_{\vec{t_0}\,}(a_1 x_1^{d}+\cdots+a_s x_s^{d}=by^{d})\cap\{(x_1,\dots,x_s,y):y=0\}|=|V_{\vec{t}\, }(a_1 x_1^{d}+\cdots+a_s x_s^{d}=0)|,$$
so that
{\small\begin{equation}\label{item315}
|V_{\vec{t_0},\, y}(a_1 x_1^{d}+\cdots+a_s x_s^{d}=by^{d})|=|V_{\vec{t_0}\,}(a_1 x_1^{d}+\cdots+a_s x_s^{d}=by^{d})|-|V_{\vec{t}\, }(a_1 x_1^{d}+\cdots+a_s x_s^{d}=0)|.
\end{equation}}
Let $\vec{a_0}=(a_1,\dots,a_s,b)$ and $\vec{d_0}=(d,\dots,d)\in\Z^{s+1}$. Since $N_s(\vec{a},\vec{d},\vec{t},q,b)=|V_{\vec{t}\, }(a_1 x_1^{d}+\cdots+a_s x_s^{d}=b)|$, $N_s(\vec{a},\vec{d},\vec{t},q,0)=V_{\vec{t}\, }(a_1 x_1^{d}+\cdots+a_s x_s^{d}=0)$ and $N_{s+1}(\vec{a_0},\vec{d_0},\vec{t_0},q,0)=V_{\vec{t_0}\,}(a_1 x_1^{d}+\cdots+a_s x_s^{d}=by^{d})$, our result follows from Eq.~\eqref{item314} and \eqref{item315} and Theorem~\ref{item303}.	$\hfill\qed$\\

\begin{remark}\label{item322}
	We observe that the condition on $t_l$ in Corollary~\ref{item309} can be relaxed to $p^{t_l}\equiv -1\pmod{d}$ and $t_l|t_i$ for all $i$ and the same proof works.
\end{remark}

\subsection{Proof of Corollary~\ref{item313}}
 Let $\Lambda=\cap_{i=1}^s\Lambda(a_i,t_i)$. By Theorem~\ref{item303}, we have that
 $$\begin{aligned}N_s(\vec{a},\vec{d},\vec{t},q,0)p^{2n-t_1-\cdots-t_s}&=\sum_{c\in\Fq^*}\prod_{i=1}^{s} \Delta(a_i,d_i,t_i,r_i,c)\\
 &=\sum_{c\in\Lambda}\prod_{i=1}^{s} \Delta(a_i,d_i,t_i,r_i,c)+\sum_{c\in\Fq\backslash\Lambda}\prod_{i=1}^{s} \Delta(a_i,d_i,t_i,r_i,c)\\
 &=|\Lambda|p^{t_1+\dots+t_s}+\sum_{c\in\Fq\backslash\Lambda}\prod_{i=1}^{s} \Delta(a_i,d_i,t_i,r_i,c),\\
 \end{aligned}$$
 and then
 \begin{equation}\label{item316}
 \begin{aligned}\left|N_s(\vec{a},\vec{d},\vec{t},q,0)p^{2n-t_1-\cdots-t_s}-|\Lambda|p^{t_1+\dots+t_s}\right|\leq \sum_{c\in\Fq\backslash\Lambda}\prod_{i=1}^{s} \Delta(a_i,d_i,t_i,r_i,c).
 \end{aligned}
 \end{equation}
 By the definition of $\Delta(a_i,d_i,t_i,r_i,c)$, it follows that
 $$|\Delta(a_i,d_i,t_i,r_i,c)|\leq \begin{cases}
 p^{t_i},&\text{ if }c\in\Lambda(a_i,t_i);\\
 d_i-1,&\text{ if }c\not\in\Lambda(a_i,t_i)\\
 \end{cases}$$
 and therefore the Triangle Inequality on Eq.~\eqref{item316} implies that
 \begin{equation}\label{item317}
 \left|\frac{N_s(\vec{a},\vec{d},\vec{t},q,0)}{ p^{t_1+\cdots+t_s-2n}}-|\Lambda|p^{t_1+\cdots+t_s}\right|\leq \sum_{c\in\Fq\backslash\Lambda}\ \prod_{i\in \nu(c)} p^{t_i}\prod_{i\not\in \nu(c)}(d_i-1),
 \end{equation}
 where $\nu(c):=\{i\in\{1,\dots,s\}:c\in\Lambda(a_i,t_i)\}$. This completes the first part of our result. Assume that $\Lambda(a_i,t_i)\cap \Lambda(a_j,t_j)=\{0\}$ for all $1\leq i< j \leq s$. In this case, we have that
 {\small$$\sum_{c\in\Fq\backslash\Lambda}\ \prod_{i\in \nu(c)} p^{t_i}\prod_{i\not\in \nu(c)}(d_i-1)
 	\leq \sum_{j=1}^s |\Lambda(a_j,t_j)|\cdot p^{t_j}\prod_{i\neq j}(d_i-1)+\Bigg(q-1-\sum_{j=1}^s(|\Lambda(a_j,t_j)|-1)\Bigg)\prod_{i=1}^s(d_i-1).$$}
 Since $|\Lambda(a_j,t_j)|=q/p^{2t_i}$, it follows that
 {\small\begin{equation}\label{item318}
 	\sum_{c\in\Fq\backslash\Lambda}\ \prod_{i\in \nu(c)} p^{t_i}\prod_{i\not\in \nu(c)}(d_i-1)\leq q\prod_{i=1}^s(d_i-1)\Bigg(\sum_{i=1}^s\bigg(\frac{1}{p^{t_i}(d_i-1)}-\frac{1}{p^{2t_i}}\bigg)+1+\frac{s-1}{q}\Bigg)
 	\end{equation}}
 Therefore, Eq.~\eqref{item317} and \eqref{item318} imply that
 $$\frac{N_s(\vec{a},\vec{d},\vec{t},q,0)}{ p^{t_1+\cdots+t_s-2n}}\geq p^{t_1+\cdots+t_s}-  q\prod_{i=1}^s(d_i-1)\Bigg(\sum_{i=1}^s\bigg(\frac{1}{p^{t_i}(d_i-1)}-\frac{1}{p^{2t_i}}\bigg)+1+\frac{s-1}{q}\Bigg),$$
 which is bigger than zero whenever
 $$\prod_{i=1}^s(d_i-1)\Bigg(\sum_{i=1}^s\bigg(\frac{1}{p^{t_i}(d_i-1)}-\frac{1}{p^{2t_i}}\bigg)+1+\frac{s-1}{q}\Bigg)<p^{t_1+\cdots+t_s-2n}.$$
 This completes the proof of the Corollary. $\hfill\qed$\\
 
 \subsection{Proof of Corollary~\ref{item311}}
 
 Let $d_1=\dots=d_s=1$. Applying Corollary~\ref{item309} by taking $r_i=t_i$, it follows that
 \begin{equation}\label{item319}
 N_s(\vec{a},\vec{d},\vec{t},q,b)=\frac{p^{t_1+\cdots+t_s-2n}}{p^{2t_l}-1}\sum_{c\in\Fq^*}\left(\Delta(-b,1,1,1,c)p^{t_l} -1\right)\prod_{i=1}^{s}\Delta(a_i,1,t_i,t_i,c),
 \end{equation}
 where
 $${\small \Delta(a_i,1,t_i,t_i,c)=\begin{cases}
 	p^{t_i}, &\text{ if }\tr_{q,p^{2t_i}}(ca_i)=0;\\
 	0, &\text{ if }\tr_{q,p^{2t_i}}(ca_i)\neq 0\\
 	\end{cases}}$$
for all $i=1,\dots,s$ and
  $${\small \Delta(-b,1,t_i,t_i,c)=\begin{cases}
 	p^{t_i}, &\text{ if }\tr_{q,p^{2t_l}}(cb)=0;\\
 	0, &\text{ if }\tr_{q,p^{2t_l}}(cb)\neq 0.\\
 	\end{cases}}$$
 We set $\lambda=|\cap_{i=1}^{s}\Lambda(a_i,t_i)|$ and $\lambda_b=|\cap_{i=1}^{s}\Lambda(a_i,t_i)\cap\Lambda(b,t_l)|$ and then our result follows from the definition of $\Lambda(a,t)$ and Eq.~\eqref{item319}. $\hfill\qed$

\section{The case $s=2$}\label{item325}

  For a curve $\Ca$ over $\Fq$, we denote by $M_n(\Ca)$ the number of rational points of $\Ca$ over $\Fqn$. For an irreducible non-singular curve $\Ca$ over $\Fq$, the well-known Riemann Hypothesis \cite[Theorem $3.3$]{moreno1993algebraic} states that the number of rational points on a curve $\Ca$ over $\Fqn$ satisfies
$$M_n(\mathcal{C})=q^n+1-\sum\limits_{i=1}^{2g} \omega_i^n,$$
where $g$ denotes the genus of $\Ca$ and $|\omega_i|=\sqrt{q}$ for all $i$. 
Also by this result, we have the well-known Hasse-Weil bound for the number of rational points on an irreducible non-singular curve over $\Fq$ of genus $g$, given by
\begin{equation}\label{item320}
|M_n(\mathcal{C})-q^n-1|\leq 2g\sqrt{q^n}.
\end{equation}
The curve $\Ca$ is called maximal over $\F_{q^2}$ if its number of points attains the Hasse-Weil upper bound, that is,
$$M_1(\Ca)=q^2+1+2gq,$$
where $g$ is the genus of $\Ca$. Similarly, a curve is called minimal over $\F_{q^2}$ if it attains the Hasse-Weil lower bound. There exist many families of maximal curves in the literature. In particular, a curve $\mathcal{J}$ with affine equation $x_1^{d}+x_2^{d}=1$ yields examples of maximal and minimal curves (e.g. see \cite{tafazolian2010characterization}). We observe that $M_n(\mathcal{J})=N_2(\vec{a},\vec{d},\vec{t},q,b)+c_0$ if $t_1=t_2=n$, where $c_0$ is the number of points at infinity on $\mathcal{J}$. Our aim in this section is to present a bound such as \eqref{item320} in the case the points on $\mathcal{C}$ has restricted solutions sets, replacing $M_n(\mathcal{C})$ by $N_2(\vec{a},\vec{d},\vec{t},q,b)$.

 \begin{corollary}\label{item321}
 	Let $q=p^{2n}$, $\vec{a}=(1,1)$, $\vec{d}=(d,d)$ and $\vec{t}=(t_1,t_2)$. Assume that there exists a divisor $r_1$ of $t_1$ and a divisor $r_2$ of $t_2$ such that $d|(p^{r_1}+1)$ and $d|(p^{r_2}+1)$. Let $t_l$ be the least integer such that $p^{t_l}\equiv -1\pmod{d}$. Then the number $N$ of points on the affine curve $x^d+y^d+1=0$ with $x\in\F_{q^{2t_1}}$ and $y\in\F_{q^{2t_2}}$ satisfies
	$$\big|N-|\Lambda|p^{2t_1+2t_2-2t}\big|\le (d-1)^2 p^{t_1+t_2}+(d-1)(p^{t_1}+p^{t_2}),$$
	where $\Lambda:=\Lambda(1,t_1)\cap\Lambda(1,t_2)$.
\end{corollary}
\begin{proof}
	Since $\Lambda(1,t_i)\subset\Lambda(1,t_l)$ for $i=1,2$, the result follows directly by Corollary~\ref{item309}.
\end{proof}

From here, pose the following problem.

\begin{problem}
	Can we obtain a bound similar to Hasse-Weil's bound for equations in two variables with restricted solutions sets?
\end{problem}

\section{The case $d_1=\dots=d_s=2$}\label{item334}

In the case where $d_1=\dots=d_s=2$, we can compute the number of solutions of diagonals equations over arbitrary finite fields, as we will see in this section. Throughout this section, we consider $q=p^n$ and we let $\tr$ denote the trace function from $\Fq$ to $\Fp$. The following result will be useful to compute the number of solutions of Eq.~\eqref{item300} in the case where $d_1=\dots=d_s=2$.

\begin{lemma}\label{item326}
	For $t|n$, $a\in\F_{q}^*$ and $\lambda\in\Fp$, let $N_\lambda$ denote the number of solutions of $\tr_{q,p}\big(ax^{2}\big)=\lambda$ over $\F_{p^{t}}$ and let $\hat{a}=\tr_{q,p^t}(a)$. Assume that $\hat{a}\neq 0$. Then
	$$N_\lambda=\begin{cases}
	p^{t-1}+ \chi(\hat{a})p^{t/2-1} ,&\text{ if }t\text{ is even}\text{ and }p\equiv 1\pmod{4};\\
	p^{t-1}+ i^t\chi(\hat{a})p^{t/2-1} ,&\text{ if }t\text{ is even}\text{ and }p\equiv 3\pmod{4};\\
	p^{t-1}+\chi(-\lambda \hat{a})p^{(t-1)/2} ,&\text{ if }t\text{ is odd}\text{ and }p\equiv 1\pmod{4},\\
	p^{t-1}+i^{t+1} \chi(-\lambda \hat{a})p^{(t-1)/2} ,&\text{ if }t\text{ is odd}\text{ and }p\equiv 3\pmod{4},\\
	\end{cases}$$
	if $\lambda\neq 0$ and
	$$N_0=\begin{cases}
	p^{t-1}-(p-1)\chi(\hat{a}) p^{t/2-1},&\text{ if }t\text{ is even}\text{ and }p\equiv 1\pmod{4};\\
	p^{t-1}-(p-1) i^t\chi(\hat{a}) p^{t/2-1},&\text{ if }t\text{ is even}\text{ and }p\equiv 3\pmod{4},\\
	p^{t-1},&\text{ if }t\text{ is odd,}
	\end{cases}$$
	where $i$ is the imaginary unity and $\chi$ denotes the quadratic multiplicative character of $\F_{p^t}$.
\end{lemma}

\begin{proof}
	Let $\psi(x)=\exp\left((2\pi i)\tr_{p^t,p}(x)/p\right)$, the canonical additive character, let $\delta\in\F_{p^t}$ be an element such that $\tr_{p^t,p}(\delta)=\lambda$ and let $\hat{a}=\tr_{q,p^t}(a)$. We observe that if $c\in\F_{p^t}$, then $\tr_{q,p}\big(ac^{2}\big)=\tr_{p^t,p}\big(\tr_{q,p^t}(a)c^{2}\big)=\tr_{p^t,p}\big(\hat{a}c^{2}\big)$ so that $\tr_{q,p}\big(ac^{2}\big)=\lambda$ if and only if $\psi(\hat{a}c^2-\delta)=1$. Therefore,
	$$\begin{aligned}pN_\lambda&=
	\sum_{c\in\F_{p^t}}\left[1+\dots+\psi(\hat{a}c^2-\delta)^{p-1}\right]\\
	&=\sum_{c\in\F_{p^t}}\left[1+\dots+\psi(\hat{a}c-\delta)^{p-1}\right]\left[1+\chi(c)\right]\\
	&=p^t+\sum_{\ell=1}^{p-1}\psi(-\delta)^\ell\chi(\ell^{-1} \hat{a}^{-1})\sum_{z\in\F_{p^t}}\psi(z)\chi(z)\\
	&=p^t+\sum_{\ell=1}^{p-1}\psi(-\delta)^\ell\chi(\ell^{-1} \hat{a}^{-1})G(\chi),\\
	\end{aligned}$$
	where $G(\chi)$ is the Gauss sum of $\chi$ over $\F_{p^t}$. Now, we split the proof into two cases: 
	\begin{itemize}
		\item If $\lambda= 0$, then
		$$\sum_{\ell=1}^{p-1}\psi(-\delta)^\ell\chi(\ell^{-1} \hat{a}^{-1})=\sum_{\ell=1}^{p-1}\chi(\ell^{-1} \hat{a}^{-1})=\begin{cases}
		(p-1)\chi(\hat{a}),&\text{ if }t\text{ is even;}\\
		0,&\text{ if }t\text{ is odd.}\\
		\end{cases}$$
		\item If $\lambda\neq 0$, then
		$$\begin{aligned}
		\sum_{\ell=1}^{p-1}\psi(-\delta)^\ell\chi(\ell^{-1} \hat{a}^{-1})&=\chi\left(\tfrac{1}{\hat{a}}\right)\sum_{\ell\in\F_p^*}\psi(-\delta)^\ell\overline\chi(\ell)=\chi(\tfrac{1}{\hat{a}})\sum_{\ell\in\F_p^*}\psi(-\delta\ell)\overline\chi(\ell)\\
		&=\chi\left(\tfrac{-\lambda}{\hat{a}}\right)\sum_{\ell\in\F_p^*}\psi(\ell)\overline\chi(\ell)=\chi\left(\tfrac{-\lambda}{\hat{a}}\right)\sum_{\ell\in\F_p^*}\psi(\ell)\chi(\ell)\\
		&=\begin{cases}
		-\chi\left(\hat{a}\right),&\text{ if }t\text{ is even;}\\
		\chi\left(-\lambda \hat{a}\right)G_p(\chi),&\text{ if }t\text{ is odd,}\\
		\end{cases}\\
		\end{aligned}$$
		where $G_p(\chi)$ is the Gauss sum of $\chi$ over $\Fp$. Our result follows directly by using Theorem 5.15 of \cite{Lidl} in the values of $G(\chi)$ and $G_p(\chi)$.
	\end{itemize}
\end{proof}

\begin{proposition}\label{item329}.
	For $c,a_1,\dots,a_s\in\Fq$, let $\hat{a}_j=\tr_{q,p^{t_j}}(ca_j)$ and let $S_{j}(c):=\sum_{x\in\scalebox{0.7}{$\displaystyle \F_{p^{t_j}}^*$}}\psi_{ca_j}\left(x^{2}\right)$. If $\hat{a}_j=0$, then $S_j(c)=p^{t_j}$. If $\hat{a}_j\neq0$, then
	$$S_{j}(c)=\begin{cases}
	-\chi^{(j)}(\hat{a}_j)p^{t_j/2} ,&\text{ if }t_j\text{ is even}\text{ and }p\equiv 1\pmod{4};\\
	- i^{t_j}\chi^{(j)}(\hat{a}_j)p^{t_j/2} ,&\text{ if }t_j\text{ is even}\text{ and }p\equiv 3\pmod{4};\\
	\chi^{(j)}(\hat{a}_j)p^{t_j/2} ,&\text{ if }t_j\text{ is odd}\text{ and }p\equiv 1\pmod{4};\\
	i^{t_j}\chi^{(j)}(\hat{a}_j)p^{t_j/2},&\text{ if }t_j\text{ is odd}\text{ and }p\equiv 3\pmod{4},\\
	\end{cases}$$
	where $i$ denotes the imaginary unity and $\chi^{(j)}$ denotes the quadratic character in $\F_{p^{t_j}}$.
\end{proposition}

\begin{proof}
	 For a fixed $j\in\{1,\dots,s\}$,  we observe that
	\begin{equation}\label{item327}
	S_{j}(c)=\sum_{\lambda\in\Fp} \exp\left(\frac{2\pi i\lambda}{p}\right)N_\lambda,
	\end{equation}
	where $N_\lambda$ is denote the number of solutions of $\tr_{q,p}\big(ax^{2}\big)=\lambda$ over $\F_{p^{t_j}}$. Assume that $t_j$ is odd. If $p\equiv 1\pmod{4}$, then Eq.~\eqref{item327} and Lemma~\ref{item326} imply that 
	$$\sum_{\lambda\in\Fp} e^{\frac{2\pi i\lambda}{p}}N_\lambda=\chi^{(j)}(\hat{a}_j)p^{\frac{t_j-1}{2}}\sum_{\lambda\in\Fp^*}e^{\frac{2\pi i\lambda}{p}}\chi^{(j)}(\lambda)=\chi^{(j)}(\hat{a}_j)p^{\frac{t_j-1}{2}}G_p(\chi).$$
	If $p\equiv 3\pmod{4}$, then Eq.~\eqref{item327} and Lemma~\ref{item326} imply that 
	$$\sum_{\lambda\in\Fp} e^{\frac{2\pi i\lambda}{p}}N_\lambda=-i^{t_j+1}\chi^{(j)}(\hat{a}_j)p^{\frac{t_j-1}{2}}\sum_{\lambda\in\Fp^*}e^{\frac{2\pi i\lambda}{p}}\chi^{(j)}(\lambda)=-i^{t_j+1}\chi^{(j)}(\hat{a}_j)p^{\frac{t_j-1}{2}}G_p(\chi).$$
	By Theorem 5.15 of \cite{Lidl}, we have that
	$$G_p(\chi)=\begin{cases}
	p^{1/2},&\text{ if }p\equiv 1\pmod{4};\\
	i p^{1/2},&\text{ if }p\equiv 3\pmod{4},\\
	\end{cases}$$
	from where our result follows. The case where $t_j$ is even follows directly from Lemma~\ref{item326}.
\end{proof}

The following result is a straightforward application of Lemma~\ref{item324} and Proposition~\ref{item329}

\begin{theorem}\label{item330}
	Let $a_1,\dots,a_s,b\in\F_q$, where $q=p^n$, and let $i$ denotes the imaginary unity. The number of solutions of the equation
	$$a_1 x_1^{2}+\cdots+a_s x_s^{2}=b$$
	with $x_j\in\F_{p^{t_j}}$ is given by
	$$p^{\tfrac{t_1+\dots+t_s}{2}-n}\sum_{c\in\scalebox{0.7}{$\displaystyle \Fq$}} \psi_c(-b)\prod_{j=1}^{s}\Gamma_j(c),$$
	where $\chi^{(j)}$ denotes the quadratic character in $\F_{p^{t_j}}$ and
	$$\Gamma_j(c)=\begin{cases}
	p^{t_j},&\text{ if }\tr_{q,p^{t_j}}(ca_j)=0;\\
	-(-1)^{t_j}\chi^{(j)}\big(\tr_{q,p^{t_j}}(ca_j)\big),&\text{ if }\tr_{q,p^{t_j}}(ca_j)\neq0\text{ and }p\equiv 1\pmod 4;\\
	-(-1)^{t_j}i^{t_j}\chi^{(j)}\big(\tr_{q,p^{t_j}}(ca_j)\big),&\text{ if }\tr_{q,p^{t_j}}(ca_j)\neq0\text{ and }p\equiv 3\pmod 4.\\
	\end{cases}$$
\end{theorem}

Theorem~\ref{item330} generalizes Theorems 6.26 and 6.27 of \cite{Lidl}.

\section{Final Comments and Open Problems}\label{item335}

This paper provided a counting on the number of solutions of diagonal equations of the form $a_1 x_1^{d_1}+\dots+a_s x_s^{d_s}=b$ with restrict solutions sets satisfying conditions on the exponents. Our counting results extend the main result of \cite{wolfmann1992number} and some results of \cite{oliveira2020diagonal} and \cite{cao2016number}. In the general setting where no conditions are imposed under the exponents, an explicit formula for such number is unknown. Indeed, the problem of counting solutions of diagonal equations in a general setting is still an open problem. In \cite{oliveira2020diagonal}, the author studies the conditions on the exponents of diagonal equations in order to find those whose number of points attains Weil's bound in the standard case where the solution sets are not restrict. From here, we propose the following problems in the restrict solutions sets case.

\begin{problem}\label{item331}
	Can one find a sharp bound by means of simple parameters, such as $d_1,\dots,d_s$ and $q$, for the number of solutions of Equation~\eqref{item300} with $x_i\in\F_{p^{t_i}}$? In this case, what are the conditions for which  diagonal equations attains such bound?
\end{problem}

From this, we have the following conjecture.

\begin{conjecture}
If there exists a bound such as proposed in Problem~\ref{item331} explicitly given in terms of $d_1,\dots,d_s$, then the diagonal equations attaining this bound are those satisfying the conditions of Theorem~\ref{item303}.
\end{conjecture}

In the light of Proposition~\ref{item308}, we wonder if it is possible to replace the monomial $x^{p^r+1}$ by polynomials of the form $xL(x)$ where $L(x)$ is a linearized polynomial in $\Fq[x]$. It is easy to verify that $\Phi_a(x)=\tr_{q,p}(axL(x))$ is a quadratic form. However with the present knowledge of quadratic forms it is not clear
whether one can obtain the dimension of the kernel of the associated symmetric bilinear form. Therefore we have the following general problem.
\begin{problem}
	Can we replace the monomials in Equation~\eqref{item300} by polynomials of the form $xL(x)$ and still obtain simple expressions for the number of solutions of these diagonal equations?
\end{problem}

\section{Acknowledgments}

I would like to thank Lucas Reis for suggesting the problem. This study was financed in part by the Coordena\c{c}\~{a}o de Aperfei\c{c}oamento de Pessoal de N\'{i}vel Superior - Brasil (CAPES) - Finance Code 001.

\printbibliography

\end{document}